\documentclass[11pt,reqno]{amsart}

\usepackage[margin=0.88in]{geometry}
\usepackage{amsmath,amssymb,amsthm}
\usepackage{mathtools}
\usepackage{microtype}

\newcommand{\IBr}{\operatorname{IBr}}
\newcommand{\Irr}{\operatorname{Irr}}
\newcommand{\Lin}{\operatorname{Lin}}
\newcommand{\Syl}{\operatorname{Syl}}
\newcommand{\core}{\operatorname{core}}

\newtheorem{theorem}{Theorem}[section]
\newtheorem{lemma}[theorem]{Lemma}
\numberwithin{equation}{section}

\title{Brauer Character Degrees and Nilpotent Subgroups}
\author{Yong Yang}
\date{}


\begin{document}

\begin{abstract}
We solve a question raised by Chen and Navarro concerning Brauer characters and nilpotent subgroups. Let $N\lhd G$, assume that $G/N$ is solvable, and let $N\le H\le G$ with $H/N$ nilpotent. We prove that for every irreducible $\ell$-Brauer character $\theta$ of $H$ there exists an irreducible $\ell$-Brauer character $\chi$ of $G$ such that $\theta$ is a constituent of $\chi_H$ and $\chi(1)$ divides $|G:H|\theta(1)$.
\end{abstract}

\maketitle

\section{Introduction}
Fix a prime $\ell$. For a finite group $X$, write $\IBr_\ell(X)$ for the set of irreducible $\ell$-Brauer characters of $X$.

Chen and Navarro proved that if $N\lhd G$, $G/N$ is solvable, and $H/N$ is abelian, then for every $\theta\in\IBr_\ell(H)$ there exists $\chi\in\IBr_\ell(G)$ such that
\[
\theta\le \chi_H \qquad\text{and}\qquad \chi(1)\mid |G:H|\theta(1).
\]
At the end of their paper they observed that the abelian hypothesis might perhaps be replaced by the assumption that $H/N$ is nilpotent, while emphasizing that such a result would require additional work~\cite{ChenNavarro}. More recently, Navarro and Tiep proved the corresponding ordinary-character theorem for nilpotent $H/N$~\cite[Theorem~5.3(a)]{NavarroTiep}. A principal ingredient in their proof is their Carter-complement extension theorem~\cite[Theorem~5.2]{NavarroTiep}.

We prove the nilpotent version suggested by Chen and Navarro. The argument follows their Brauer-character reduction. The final step uses Navarro--Tiep's ordinary Carter-subgroup theorem~\cite[Theorem~5.2(b)]{NavarroTiep}. When the chief-factor prime differs from $\ell$, the relevant normal subgroup is an $\ell'$-group, and Lemma~\ref{lem:reduction} below shows directly that the ordinary character obtained from Navarro--Tiep has irreducible reduction modulo $\ell$.

\begin{theorem}\label{thm:main}
Let $G$ be a finite group, let $N\lhd G$, assume that $G/N$ is solvable, and let $N\le H\le G$ such that $H/N$ is nilpotent. If $\theta\in\IBr_\ell(H)$, then there exists $\chi\in\IBr_\ell(G)$ such that
\[
\theta\le \chi_H \qquad\text{and}\qquad \chi(1)\mid |G:H|\theta(1).
\]
\end{theorem}

We separate two elementary points in order to make the passage between Brauer and ordinary characters explicit. The first is the relative monomial reduction needed when the subgroup becomes nilpotent. The second shows why, in the final $\ell'$-chief-factor case, the ordinary character supplied by Navarro--Tiep has irreducible reduction modulo $\ell$.

\section{Two preliminary lemmas}

\begin{lemma}\label{lem:monomial}
Let $X$ be a finite nilpotent group, let $Z\le Z(X)$ be an $\ell'$-subgroup, and let $\alpha\in\IBr_\ell(X)$. Assume that
\[
\alpha_Z=\alpha(1)\zeta
\]
for some $\zeta\in\IBr_\ell(Z)$. If $\alpha(1)>1$, then there exist $Z\le U<X$ and a linear $\lambda\in\IBr_\ell(U)$ such that $\lambda^X=\alpha$.
\end{lemma}

\begin{proof}
Write $X=P\times R$, where $P\in\Syl_\ell(X)$ and $R$ is the Hall $\ell'$-subgroup of $X$. Since every simple module of the $\ell$-group $P$ in characteristic $\ell$ is trivial, $P\le\ker(\alpha)$, and $\alpha$ is the inflation of an ordinary irreducible character $\beta\in\Irr(R)$. As $Z$ is an $\ell'$-subgroup of the center of $X$, we have $Z\le R$.

Since $R$ is nilpotent, it is supersolvable. By the classical monomiality theorem for supersolvable groups~\cite[Theorem~6.22]{Isaacs}, every irreducible character of $R$ is monomial. Hence
\[
\beta=\mu^R
\]
for some subgroup $S\le R$ and some $\mu\in\Lin(S)$. We claim that $Z\le S$. If not, choose $z\in Z\setminus S$. Since $z$ is central in $R$, no conjugate of $z$ lies in $S$, and the formula for an induced character gives $\beta(z)=0$. On the other hand,
\[
\beta_Z=\beta(1)\zeta,
\]
so $\beta(z)=\beta(1)\zeta(z)\ne0$, a contradiction. Thus $Z\le S$.

Set $U=P\times S$ and inflate $\mu$ to the linear Brauer character $\lambda$ of $U$ that is trivial on $P$. Then $Z\le U$ and $\lambda^X=\alpha$. Since $\alpha(1)>1$, the inducing subgroup is proper.
\end{proof}

\begin{lemma}\label{lem:reduction}
Let $K\le G$ be an $\ell'$-group and let $\eta\in\Irr(K)$. Suppose that $\widehat\chi\in\Irr(G)$ satisfies
\[
\widehat\chi_K=\eta.
\]
Then $\widehat\chi^{\,0}$ is an irreducible $\ell$-Brauer character of $G$.
\end{lemma}

\begin{proof}
Write
\[
\widehat\chi^{\,0}=\sum_i d_i\varphi_i,\qquad d_i>0,\quad \varphi_i\in\IBr_\ell(G).
\]
Every element of $K$ is $\ell$-regular, and hence
\[
\eta=(\widehat\chi^{\,0})_K=\sum_i d_i(\varphi_i)_K.
\]
For each $i$, the restriction $(\varphi_i)_K$ is a nonzero Brauer character of $K$ and, since $K$ is an $\ell'$-group, it is an ordinary character of $K$ with nonnegative integral irreducible multiplicities. The right side is therefore a nonnegative integral sum of ordinary irreducible characters of $K$, whereas the left side is the single irreducible character $\eta$. It follows that there is exactly one summand, with coefficient one, and that its restriction to $K$ is $\eta$. Hence $\widehat\chi^{\,0}$ is irreducible.
\end{proof}

\section{Proof of the theorem}

\begin{proof}[Proof of Theorem~\ref{thm:main}]
We follow the induction used by Chen and Navarro in the proof of~\cite[Theorem~A]{ChenNavarro}, replacing only the final ordinary-character step. We argue first on $|G:N|$ and, for fixed $|G:N|$, on $|G:H|$.

Choose $\nu\in\IBr_\ell(N)$ under $\theta$, and let $T=G_\nu$. Suppose first that $T<G$. Let $\epsilon\in\IBr_\ell(T\cap H)$ be the Clifford correspondent of $\theta$ over $\nu$. Since $(T\cap H)/N$ is a subgroup of the nilpotent group $H/N$, it is nilpotent, while $T/N$ is solvable. By induction there is $\xi\in\IBr_\ell(T)$ such that
\[
\epsilon\le \xi_{T\cap H},\qquad
\xi(1)\mid |T:T\cap H|\epsilon(1).
\]
Because $\xi$ lies over $\nu$, Brauer Clifford correspondence gives $\chi=\xi^G\in\IBr_\ell(G)$. Moreover $T\cap H=H_\nu$ and $\theta=\epsilon^H$. In the Mackey decomposition of $(\xi^G)_H$, the double coset $TH$ contributes $(\xi_{T\cap H})^H$, which contains $\epsilon^H=\theta$. Hence $\theta\le\chi_H$. Finally,
\[
\chi(1)=|G:T|\xi(1)
\mid |G:T\cap H|\epsilon(1)
=|G:H|\theta(1).
\]
Thus we may assume that $\nu$ is $G$-invariant.

Apply the Brauer character-triple reduction used in the proof of~\cite[Theorem~A]{ChenNavarro}. It replaces the triple by an isomorphic one in which
\[
N\le Z(G),\qquad \ell\nmid |N|,
\]
and $\nu$ is faithful. The isomorphism identifies all subgroups above $N$ through the common quotient and preserves the relevant character degrees, induction, and restriction. In particular the subgroup corresponding to $H$ has quotient isomorphic to $H/N$, so nilpotence is preserved.

We may also assume
\begin{equation}\label{eq:core}
\core_G(H)=N.
\end{equation}
Indeed, if $M=\core_G(H)>N$, then $M\lhd G$, $M\le H$, $G/M$ is solvable, and $H/M$ is a quotient of the nilpotent group $H/N$. The induction hypothesis applied to $(G,M,H,\theta)$ gives the theorem immediately, since $|G:M|<|G:N|$.

Since $N$ is central, it is abelian; hence the assumed solvability of $G/N$ also implies that $G$ is solvable. In particular $H$ is nilpotent: if $H/N$ has nilpotency class $c$, then $\gamma_{c+1}(H)\le N\le Z(H)$, and hence $\gamma_{c+2}(H)=1$.

Let $K/N$ be a chief factor of $G$. We may assume
\begin{equation}\label{eq:GKH}
G=KH.
\end{equation}
Indeed, suppose that $KH<G$. Since $KH/N$ is solvable and $H/N$ is nilpotent, induction applied to $(KH,N,H,\theta)$ gives $\tau\in\IBr_\ell(KH)$ with
\[
\theta\le\tau_H,\qquad \tau(1)\mid |KH:H|\theta(1).
\]
Now $G/K$ is solvable and
\[
KH/K\cong H/(H\cap K)
\]
is a quotient of $H/N$, hence nilpotent. A second application of induction, to $(G,K,KH,\tau)$, gives $\chi\in\IBr_\ell(G)$ with
\[
\tau\le\chi_{KH},\qquad \chi(1)\mid |G:KH|\tau(1).
\]
Then $\theta\le\chi_H$ and
\[
\chi(1)\mid |G:KH||KH:H|\theta(1)=|G:H|\theta(1),
\]
so the theorem follows. Thus, for the remaining case,~\eqref{eq:GKH} holds. Also
\begin{equation}\label{eq:intersection}
K\cap H=N.
\end{equation}
To see this, $K\cap H$ is normalized by $H$. If $x\in K\cap H$ and $k\in K$, then $[x,k]\in N$ because $K/N$ is abelian; since $N\le K\cap H$, we have $x^k\in K\cap H$. Thus $K\cap H\lhd G=KH$, and~\eqref{eq:core} yields~\eqref{eq:intersection}.

Write $K/N$ as an elementary abelian $q$-group. Since $K/N$ is a chief factor and $G=KH$, the action of $H/N$ on $K/N$ is irreducible. Its kernel is
\[
C_H(K/N)/N.
\]
Set $C=C_H(K/N)$. The subgroup $C$ is normalized by $H$. Also $[C,K]\le N$, so $K$ normalizes $C$. Hence $C\lhd G=KH$, and~\eqref{eq:core} implies $C=N$. Thus the action is faithful. If a Sylow $q$-subgroup $Q/N$ of the nilpotent group $H/N$ were nontrivial, then $Q/N\lhd H/N$. A nontrivial $q$-group acting on the $\mathbb F_q$-space $K/N$ has nonzero fixed points. Since $Q/N$ is normal in $H/N$, the fixed-point space $C_{K/N}(Q/N)$ is $H/N$-invariant. It is nonzero, so irreducibility forces $C_{K/N}(Q/N)=K/N$; hence $Q/N$ acts trivially, contrary to faithfulness. Thus
\begin{equation}\label{eq:coprime}
q\nmid |H/N|.
\end{equation}
This coprimeness is used below to choose an $H$-invariant character above $\widehat\nu$.

We next reduce to the case that $\theta$ is linear. Since $N\le Z(H)$ and $\nu$ lies under $\theta$,
\[
\theta_N=\theta(1)\nu.
\]
Suppose that $\theta(1)>1$. Lemma~\ref{lem:monomial} gives $N\le U<H$ and a linear $\tau\in\IBr_\ell(U)$ such that $\tau^H=\theta$. Since $N\le U$ and $K\cap H=N$, we have
\[
KU\cap H=U(K\cap H)=U.
\]
Thus $KU<G$, since otherwise $H=KU\cap H=U$, contrary to $U<H$. Apply induction first to the quadruple $(KU,N,U,\tau)$. Thus there is $\rho\in\IBr_\ell(KU)$ such that $\tau\le\rho_U$ and
\[
\rho(1)\mid |KU:U|.
\]
Apply induction again to the quadruple $(G,K,KU,\rho)$. Since
\[
KU/K\cong U/N
\]
is nilpotent and $G/K$ is solvable, there is $\chi\in\IBr_\ell(G)$ such that $\rho\le\chi_{KU}$ and
\[
\chi(1)\mid |G:KU|\rho(1).
\]
Since $\rho\le\chi_{KU}$ and $\tau\le\rho_U$, we have $\tau\le\chi_U$. We now use the special form of the subgroup $U$ supplied by Lemma~\ref{lem:monomial}. Write $H=P\times R$, where $P\in\Syl_\ell(H)$ and $R$ is the Hall $\ell'$-subgroup. In the proof of Lemma~\ref{lem:monomial}, one has $U=P\times S$ for some $S\le R$, and $\tau$ is the inflation of a linear character $\mu\in\Irr(S)$ such that $\mu^R=\beta$, where $\beta\in\Irr(R)$ corresponds to $\theta$. Since $S$ is an $\ell'$-group, the restriction $\chi_S$ is an ordinary character, and $\tau\le\chi_U$ implies $\mu\le\chi_S$. Ordinary Frobenius reciprocity now gives
\[
0< \langle\chi_S,\mu\rangle_S=\langle\chi_R,\mu^R\rangle_R=\langle\chi_R,\beta\rangle_R.
\]
Thus $\beta$ is a constituent of $\chi_R$. As the irreducible Brauer characters of $H=P\times R$ are precisely the inflations of the irreducible ordinary characters of $R$, it follows that $\theta\le\chi_H$. Moreover,
\[
\chi(1)\mid |G:KU||KU:U|=|G:H||H:U|=|G:H|\theta(1).
\]
This proves the theorem in this case. Hence from now on
\begin{equation}\label{eq:linear}
\theta(1)=1.
\end{equation}

Suppose first that $q=\ell$. Since $N\le Z(K)$ is an $\ell'$-Hall subgroup of $K$, Schur--Zassenhaus gives a Sylow $\ell$-subgroup $Q$ with
\[
K=N\times Q.
\]
Every irreducible $\ell$-Brauer character of $K=N\times Q$ has $Q$ in its kernel. Hence there is a unique one lying over $\nu$, namely the extension $\delta$ defined by $\delta(nq)=\nu(n)$ for $n\in N$ and $q\in Q$. Because $\nu$ is $G$-invariant and $K\lhd G$, this uniqueness makes $\delta$ $G$-invariant. The restriction correspondence of Guralnick--Navarro~\cite[Lemma~5.2]{GuralnickNavarro}, used in this form by Chen--Navarro in the proof of~\cite[Theorem~A]{ChenNavarro}, then gives a bijection
\[
\IBr_\ell(G\mid\delta)\longrightarrow\IBr_\ell(H\mid\nu).
\]
Thus some $\chi\in\IBr_\ell(G)$ satisfies $\chi_H=\theta$. Hence $\chi(1)=\theta(1)=1$, and therefore $\chi(1)\mid |G:H|\theta(1)$.

Assume therefore that
\begin{equation}\label{eq:qnotell}
q\ne\ell.
\end{equation}
Since $N$ is an $\ell'$-group, $K$ is then an $\ell'$-group.

Because $H$ is nilpotent and $\theta$ is linear, $\theta$ has a linear ordinary lift. Explicitly, write $H=P\times R$, with $P\in\Syl_\ell(H)$ and $R$ the Hall $\ell'$-subgroup. The Brauer character $\theta$ is trivial on $P$ and corresponds to a linear ordinary character of $R$; extending it trivially across $P$ gives
\[
\widehat\theta\in\Lin(H),\qquad \widehat\theta^{\,0}=\theta.
\]
Put $\widehat\nu=\widehat\theta_N$. Since $N$ is an $\ell'$-group, $\widehat\nu$ is the ordinary linear character corresponding to $\nu$. Also $\widehat\nu$ is $H$-invariant, because $N\lhd H$ and $\widehat\theta$ is linear.

At this point we use the same coprime-action step as Navarro--Tiep. By~\eqref{eq:coprime}, the $q$-group $K/N$ and the group $H/N$ have coprime orders. Since $\widehat\nu$ is $H$-invariant, the coprime-action invariant-character theorem~\cite[Theorem~13.31]{Isaacs} yields an $H$-invariant character
\[
\eta\in\Irr(K\mid\widehat\nu).
\]
This is exactly the choice made in the proof of~\cite[Theorem~5.3(a)]{NavarroTiep}. Since $G=KH$, the $H$-invariance of $\eta$ makes $\eta$ $G$-invariant.

We claim that $H$ is maximal in $G$. Let $H\le L\le G$. Since $G=KH$, every element of $L$ can be written with its $H$-part in $L$, and hence
\[
L=H(L\cap K).
\]
Now $(L\cap K)/N$ is $H/N$-invariant in $K/N$. By irreducibility it is either trivial or all of $K/N$, so $L=H$ or $L=G$. Thus $H$ is maximal in $G$. If $H\lhd G$, then~\eqref{eq:core} gives $H=N$. Since then $H/N=1$ is abelian, the result follows directly from Chen--Navarro~\cite[Theorem~A]{ChenNavarro}. We may therefore assume that $H$ is not normal. Since $H$ is maximal in $G$, the strict inclusion $H<N_G(H)$ would force $N_G(H)=G$ and hence $H\lhd G$, a contradiction. Therefore $N_G(H)=H$. As $N\lhd G$ and $N\le H$, this gives
\[
N_{G/N}(H/N)=H/N.
\]
Thus $H/N$ is a self-normalizing nilpotent subgroup, i.e. a Carter subgroup, of the solvable group $G/N$.

We now have all the hypotheses of Navarro--Tiep~\cite[Theorem~5.2(b)]{NavarroTiep}: $G$ is solvable, $G=KH$, $K\cap H=N$, the section $K/N$ is abelian, $H/N$ is nilpotent and self-normalizing in $G/N$, $\eta\in\Irr(K)$ is $H$-invariant, and $\widehat\nu\in\Irr(N)$ is an $H$-invariant constituent of $\eta_N$. We apply that theorem with $L=N$ and with the extension $\widehat\theta\in\Irr(H)$ of $\widehat\nu$. The theorem yields an irreducible constituent
\[
\widehat\chi\in\Irr(G)
\]
of $\widehat\theta^{\,G}$ which extends $\eta$. Therefore
\begin{equation}\label{eq:extension}
\widehat\chi_K=\eta,\qquad \widehat\chi(1)=\eta(1).
\end{equation}
Since $N\lhd K$ and $\eta\in\Irr(K\mid\widehat\nu)$, the standard subnormal degree-divisibility theorem~\cite[Corollary~11.29]{Isaacs} gives
\[
\frac{\eta(1)}{\widehat\nu(1)}\mid |K:N|.
\]
As $\widehat\nu$ is linear and $\widehat\chi$ extends $\eta$, we obtain
\begin{equation}\label{eq:degree}
\eta(1)=\widehat\chi(1)\mid |K:N|.
\end{equation}
By Lemma~\ref{lem:reduction},
\[
\chi:=\widehat\chi^{\,0}\in\IBr_\ell(G)
\]
is irreducible. Since $\widehat\chi$ is a constituent of $\widehat\theta^{\,G}$, ordinary Frobenius reciprocity gives $\widehat\theta\le\widehat\chi_H$. Reducing on $\ell$-regular elements and using $\widehat\theta^{\,0}=\theta$ gives
\[
\theta\le\chi_H.
\]
Finally, by~\eqref{eq:intersection}, \eqref{eq:linear}, \eqref{eq:extension}, and~\eqref{eq:degree},
\[
\chi(1)=\widehat\chi(1)=\eta(1)\mid |K:N|=|G:H|=|G:H|\theta(1).
\]
This completes the proof.
\end{proof}

\section*{Acknowledgements}
This work was partially supported by a grant from the Simons Foundation (\#918096, to YY).

\section*{Disclosure Statement}
The authors declare that they have no competing interests and no conflicts of interest.

\section*{Data Availability Statement}
Data sharing is not applicable to this article, as no data sets were generated or analysed during the current study.

\bigskip
\noindent Department of Mathematics, Texas State University, 601 University Drive, San Marcos, TX 78666, USA

\noindent Email address: \texttt{yang@txstate.edu}


\begin{thebibliography}{99}

\bibitem{ChenNavarro}
X. Chen and G. Navarro,
\emph{Brauer characters, degrees and subgroups},
Bull. Lond. Math. Soc. \textbf{54} (2022), 891--893.

\bibitem{GuralnickNavarro}
R. Guralnick and G. Navarro,
\emph{Real constituents of permutation characters},
J. Algebra \textbf{607} (2022), 315--337.

\bibitem{Isaacs}
I. M. Isaacs,
\emph{Character Theory of Finite Groups},
AMS Chelsea Publishing, Providence, RI, 2006.

\bibitem{NavarroTiep}
G. Navarro and P. H. Tiep,
\emph{Degrees and constituents of characters},
J. Algebra \textbf{689} (2026), 423--453.

\end{thebibliography}
\end{document}